\documentclass[10pt,reqno,letter]{article}

\newcommand {\level}[1]   {^{(#1)}}

\usepackage{mathtools}
\usepackage{pifont}

\usepackage{pdflscape}

\usepackage{extsizes}
\usepackage{blindtext}
\usepackage[T1]{fontenc}
\usepackage{amsthm}
\usepackage[utf8]{inputenc}
\usepackage[english]{babel}
\newcommand{\dd}{{\mathrm{d}}}
\newcommand{\cc}{{\mathsf{C}}}

\newcommand{\dist}{{\mathrm{dist}}}
\newcommand{\diam}{{\mathrm{diam}}}
\newcommand{\radius}{{\mathrm{radius}}}
\newcommand{\pa}{{\mathsf{Alice}}}
\newcommand{\pb}{{\mathsf{Bob}}}
\newcommand{\Leb}{{\mathsf{Leb}}}
\usepackage{hyperref}

\usepackage{tikz-cd}
\usepackage{pgfplots}
\usetikzlibrary{babel}
\usepackage{float}
\usepackage{amsthm}
\usepackage{amsmath}\usepackage{amssymb}
\usepackage{commath}
\newtheorem{theorem}{Theorem}
\newtheorem{corollary}{Corollary}
\newtheorem{question}{Question}
\newtheorem{lemma}{Lemma}
\newtheorem{proposition}{Proposition}

\theoremstyle{definition}
\newtheorem{hypothesis}{Statement}
\newtheorem*{inductive}{Inductive Hypothesis}

\theoremstyle{definition}

\newtheorem{system}{System}

\theoremstyle{remark}
\newtheorem*{remark}{Remark}
\usepackage{pst-node}
\usepackage{tikz-cd}

  \usepackage{tcolorbox}
  \usetikzlibrary{arrows}
\usetikzlibrary{shapes}

\usepackage{tikz}
\usetikzlibrary{decorations.markings}
\usepackage{framed}
\usetikzlibrary{patterns}

\usepackage{tikz}

            \usetikzlibrary{decorations.markings}
\usepackage{pst-plot}
\usepackage{tikz-cd}
\usepackage{pgfplots}

\usepackage{pgffor}

\usepackage{enumerate}
\title{McMullen's game for equicontinuously-twisted badly approximable points in continued fractions and beta expansions}
\title{McMullen’s game for equicontinuously-twisted badly
approximable points in continued fractions and beta expansions}
\author{David Lambert \\
University of North Texas, 1155 Union Cir
\#311430\\ Denton, TX 76203-5017.
\and David Simmons \\
434 Hanover Ln, \\ Irving, TX 75062.
\and Jiajie Zheng \\
University of North Texas, 1155 Union Cir
\#311430\\ Denton, TX 76203-5017.} 

\usetikzlibrary{decorations.pathreplacing,decorations.markings,calc,arrows.meta,bending}

\usepackage{circuitikz}

\usepackage{pgfplots}

  \usetikzlibrary{shapes.geometric}
\usetikzlibrary{angles,quotes}
\begin{document}

\newcommand{\sss}{{\mathsf{s}}}
\newcommand{\eee}{{\mathsf{e}}}
\newcommand{\wa}{{\mathcal{W}}}
\newcommand{\intt}{{\mathsf{Int}}}
\newcommand{\cb}{{\overline{B}}}

\maketitle

\begin{abstract}
    In a $\beta$-transformation system for an integer $\beta>0$ or the Gauss map system, given a sequence of functions $\left(f_n:[0,1]\to [0,1]\right)_{n=1}^\infty$ consider the set 
    \begin{equation*}
        \mathfrak{BA}\left(\left(f_n\right)_n\right):=\left\{x\in X:\, \liminf_{n\to\infty}\dist\left(T^n x,f_n(x)\right)>0\right\},
    \end{equation*}
    which generalizes the nonrecurrent set (by taking $f_n$'s to be identity functions) and the non-shrinking-target set (by taking $f_n$'s to be constant functions, which generalizes the set of badly approximable numbers). It has been shown if $\left(f_n\right)_n$ is a sequence of the same constant functions, then $\mathfrak{BA}\left(\left(f_n\right)_n\right)$ is winning in the sense of McMullen's game, which implies it is dense and full Hausdorff dimensional everywhere. In this paper, we strengthen this result by extending $\left(f_n\right)_n$ to be any equicontinuous sequence. 
\end{abstract}

\setcounter{section}{-1}

\section{Introduction}

\textsl{Poincar\'{e} Recurrence Theorem} (see e.g. \cite[Theorem 2.11]{ergodictheorynum}) is one of the most fundamental results in ergodic theory, and it asserts that for a metric measure preserving system $(X,\dist,\mathcal{B},\mu,T)$ (i.e., $(X,\dist)$ is a metric space, $\mathcal{B}$ is a Borel $\sigma$-algebra of the metric space $X$, $T:X\to X$ is a measurable function $\mu$ is a $T$-invariant probability measure) with countable base, $\mu$-almost all points return close to themselves under the iterations of $T$ in the sense that 
\begin{equation*}
    \mu\left(\mathfrak{R}\right)=1,\quad \text{where } \mathfrak{R}:=\left\{x\in X: \liminf_{n\to\infty}\dist \left(T^nx,x\right)=0\right\}
\end{equation*}
(here and hereafter $\dist\left(\circ,\circ\right)$ denotes the usual Euclidean distance when $X=[0,1]$). 

A closely related topic is the \textsl{shrinking target problem}, which studies the orbits accumulating near a given point. In any metric measure preserving system $(X,\dist,\mathcal{B},\mu,T)$ where $T$ is ergodic, by Birkhoff's Ergodic Theorem (see e.g. \cite[Theorem 2.30]{ergodictheorynum}), for any fixed target $y\in X$, the orbit of a $\mu$-generic point accumulates around $y$; namely, \begin{equation}\label{shrinkingtargetresult}
  \forall\, y\in X,\quad   \mu\left(\mathfrak{A}(y)\right)=1,\quad \text{where }\mathfrak{A}(y):=\left\{x\in X: \liminf_{n\to\infty}\dist \left(T^nx,y\right)=0\right\}. 
\end{equation}  
In particular, in \cite{metricnumber} Philipp\footnote{The results in \cite{metricnumber} are stronger; it concerns sets $\overline{\mathfrak{A}}(y,\psi):=\limsup_{n\to\infty}\left\{x\in X:\, \dist \left(T^nx,y\right)<\psi(n)\right\}$ for a positive sequence $\psi$. An analogous results for recurrence concerning the sets $\overline{\mathfrak{R}}(\psi):=\limsup_{n\to\infty}\left\{x\in X:\, \dist \left(T^nx,x\right)<\psi(n)\right\}$ can be found in \cite{simmonsrecurrence} and that for twisted recurrence can be found in \cite{kz}.} focused on the shrinking targets in the following systems 
\renewcommand{\thesystem}{\Roman{system}}
\begin{framed}
\begin{system}
 Let $\gamma>1$.   $T:[0,1]\to [0,1],x\mapsto \gamma x\mod{1}$. 
\end{system}

\begin{system}
    $T:[0,1]\to [0,1], x\mapsto \begin{cases}
        0 & \text{if }x=0,\\
        \frac{1}{x} \mod{1} &\text{otherwise}. 
    \end{cases}$
\end{system}
\end{framed}
\noindent with the usual distances and invariant Borel measures $\mu$. 
The sets $\mathfrak{A}(y)$ of points whose orbits accumulate around a fixed point $y$ are of interests in Diophantine number theory. For example, it is well known (see e.g. \cite[Theorem 23]{khinchincontinuedfrac} that the set of \textsl{badly approximable numbers}
\begin{equation*}
    \mathfrak{BA}:=\left\{x\in [0,1]:\, \exists c>0 \text{ such that }\left|x-\frac{p}{q}\right|>\frac{c}{q^2}\, \forall q\in \mathbb{N}, p\in \mathbb{Z}\right\}
\end{equation*}
is equal to the set of numbers whose partial denominators in their continued fraction expansions are bounded, namely $\mathfrak{A}(0)^\cc$ in System II; i.e., 
\begin{equation*}
    \mathfrak{BA}=\left\{x\in X: \liminf_{n\to\infty}\dist \left(T^nx,0\right)>0\right\}=\mathfrak{A}(0)^\cc, \quad \text{where }(X,T)\text{ is defined as in System II}. 
\end{equation*}
Moreover, let $\gamma>1$ be an integer. A real number $x$ is called \textsl{normal} to base $\gamma$ if 
for any interval $I\subset [0,1]$ with $\gamma$-adic rational endpoints and the Lebesgue measure $\Leb$, 
\begin{equation*}
    \lim_{n\to \infty}\frac{1}{n}\left|\left\{0\leq k\leq n-1:\, T^k x\in I\right\}\right|=\Leb(I),\quad \text{where }(X,T)\text{ is defined as in System I}. 
\end{equation*}
Equivalently, a number $x$ is normal to $\gamma$ iff every word of $n$ digits from $\{0,1,\ldots,b-1\}$ occurs in the $\gamma$-expansion of $x$ with asymptotic frequency of $\frac{1}{\gamma^n}$ as $n\to \infty$. In System I, every number in $\mathfrak{A}(0)^\cc$ has a uniform bound on the number of consecutive digit-$0$ in its base-$\gamma$ expansion. Clearly, such numbers are not normal, so $\mathfrak{A}(0)^\cc \subset \mathfrak{BA}$.\par 
Additionally, classic results (see \cite{schmidtdiophan}) have shown that  
\begin{theorem}\label{classicthm1}
\begin{enumerate}
    \item The Hausdorff dimension of $\mathfrak{A}(y)^\cc$ in System I is equal to $1$ for all $y\in [0,1]$
    \item The Hausdorff dimension of $\mathfrak{A}(0)^\cc$ in System II is equal to $1$.
\end{enumerate}
\end{theorem}
 In particular, Theorem \ref{classicthm1} implies that $\mathfrak{BA}$ and the collection of numbers not normal to a fixed integer base $\gamma$ have full Hausdorff dimension. \par 
 An upgrade to Theorem \ref{classicthm1} is to consider the subset's winning/losing property in \textsl{Schmidt's game}, which was introduced in \cite{schmidtsgame}. Schmidt showed in \cite{schmidtsgame} that 
\begin{theorem}\label{classicthm2}
\begin{enumerate}
    \item For all integers $\gamma>1$, the set $\mathfrak{A}(y)^\cc$ in System I is $\alpha$-winning for all $0<\alpha<1/2$. 
    \item The set $\mathfrak{A}(0)^\cc$ in System II is $\alpha$-winning for all $0<\alpha<1/2$.
\end{enumerate}
\end{theorem}
We will define $\alpha$-winning in Section \ref{secgame}. This property is strictly stronger than full Hausdorff dimension, see \cite{schmidtsgame}, so Theorem \ref{classicthm2} strengthens Theorem \ref{classicthm1}. Note the fact proved by Schmidt \cite{schmidtsgame} that a countable intersection of $\alpha$-winning sets is also $\alpha$-winning, so the collection of numbers normal to no base is also $\alpha$-winning for all $0<\alpha<1/2$ and in particular it has Hausdorff dimension $1$ as a consequence.

A modification of Schmidt's game was introduced by McMullen in \cite{mcmullengame} which is later been refered to as \textsl{McMullen's game}. 
\begin{theorem}\label{classicthm3}
\begin{enumerate}
    \item For all integers $\gamma>1$, the set $\mathfrak{A}(y)^\cc$ in System I is $\beta$-absolute winning for all $\beta\in \left(0,\frac{1}{3}\right)$. 
    \item The set $\mathfrak{A}(0)^\cc$ in System II is $\beta$-absolute winning for all $\beta\in \left(0,\frac{1}{3}\right)$.
\end{enumerate}
\end{theorem}
We will define $\beta$-absolute-winning in Section \ref{secgame}. $\beta$-absolute-winning for all $\beta\in \left(0,\frac{1}{3}\right)$ is strictly stronger than $\alpha$-winning for all $\alpha \in \left(0,\frac{1}{2}\right)$, see \cite{mcmullengame}, so Theorem \ref{classicthm3} strengthens Theorem \ref{classicthm2}.

Given a metric measure preserving system $(X,\dist,\mathcal{B},\mu,T)$, an effort to unify the settings of recurrence and shrinking targets is to consider a \textsl{twist}, namely a Borel measurable function $f:X\to X$. For a sequence of functions $\left(f_n:X\to X\right)_n$, define the set of \textsl{$\left(f_n\right)_n$-twisted-badly-approximable} points for $T$ by 
\begin{equation*}
\mathfrak{BA}\left(\left(f_n\right)_n\right):=\left\{x\in X:\, \liminf_{n\to\infty}\dist \left(T^n x,f_n(x)\right)>0\right\}. 
\end{equation*}
The recurrence and shrinking targets settings correspond to the twists $\left(f_n:X\to X\right)_n$ being the identity functions and constant functions respectively. 
It was shown in \cite[Corollary 2.4]{kz} that if $T$ is ergodic  and $\mu$ has full support, then $\mathfrak{BA}\left(\left(f_n\right)_n\right)$ has $\mu$-measure $0$. However, our result will show that $\mathfrak{BA}\left(\left(f_n\right)_n\right)$ is ``big'' in terms of its fractal dimensions and games, despite being null in invariant measures.  Here is our main result: 
\begin{theorem}\label{mainthm}
    In Systems I and II, for any equicontinuous sequence of functions $\left(f_n:[0,1]\to [0,1]\right)_n$, the set $\mathfrak{BA}\left(\left(f_n\right)_n\right)$ is $\beta$-absolute winning for all $\beta\in \left(0,\frac{1}{3}\right)$. 
\end{theorem}

The structure of the paper is as follows. In Section \ref{secgame}, we will introduce Schimdt's game and McMullen's game, and some important properties of winning sets in either game. In Section \ref{secproof}, we will prove the main result Theorem \ref{mainthm} by induction: in Section \ref{secproof1}, we will set up the notations and definitions that we will use throughout the proof, and the inductive hypotheses will be established; in Section \ref{secproof2} and \ref{secproof3}, we will construct the strategies for $\beta$-winning in McMullen's game for all $\beta\in \left(0,\frac{1}{3}\right)$ for System I and System II respectively, and the inductive hypothese will be verified; in Section \ref{finishproof}, we will finish the proof. 

\section{Schmidt's game and McMullen's absolute game}\label{secgame}

For this article to be self-contained, we define \textsl{Schmidt's game} and \textsl{McMullen's game}, which are used to state Theorems \ref{classicthm2} and \ref{classicthm3} and our main result Theorem \ref{mainthm}. Let $(X,d)$ be a complete metric space.\par 
For any pair $(z,\rho)\in X\times (0,\infty)$, define the \textsl{closed ball} $B(z,\rho)$ to be the set
\begin{equation*}
    B(z,\rho):=\left\{x\in X:\, d(z,r)\leq \rho\right\}. 
\end{equation*}\par 
 Schmidt's game was first introduced by Schmidt in \cite{schmidtsgame} and since then this infinite game has been used as a notion of fractal dimensions (for example, see \cite{danischmidt,toralgame,modifiedscgame,dkwbschmidtgame,normaltonobase,fractalgame,markovgame}).  To define Schmidt's game, let $\alpha,\beta\in (0,1)$ be constants. Consider two players $\pa$ and $\pb$. $\pb$ and $\pa$ take turns to pick elements $(z_k,r_k)$ and $(\zeta_k,\rho_k)$ of $X\times (0,\infty)$ respectively, such that 
\begin{equation*}
    B(z_1,r_1)\supset B(\zeta_1,\rho_1)\supset B(z_2,r_2) \supset \cdots B(z_k,r_k)\supset B(\zeta_k,\rho_k)\supset B(z_{k+1},r_{k+1})\supset \cdots 
\end{equation*}
and 
\begin{equation*}
    \rho_k=\alpha r_k\quad\text{and}\quad r_{k+1}=\beta \rho_k. 
\end{equation*}
A subset $S$ of $X$ is called 
\begin{itemize}
    \item \textsl{$(\alpha,\beta)$-winning} given constants $\alpha,\beta\in (0,1)$ if $\pa$ can always set up her plays, regardless of how $\pb$ plays, so that $\bigcap_{k=1}^\infty B(z_k,r_k)\cap S\neq \varnothing$, and
    \item \textsl{$\alpha$-winning} given a constant $\alpha\in (0,1)$ if $S$ is $(\alpha,\beta)$-winning for all $\beta\in (0,1)$, and 
    \item \textsl{winning} if $S$ is $\alpha$-winning for some $\alpha\in (0,1)$. 
\end{itemize}
It is easy to observe that 
\begin{proposition}\label{secgameprop1}
    If $S\subset X$ is winning, then $S$ is dense in $X$. 
\end{proposition}
$\alpha$-winning in Schmidt's game is stable under taking countable intersections; specifically, 
\begin{proposition}\label{secgameprop2}
    For all $\alpha\in (0,1)$, the countable intersection of $\alpha$-winning sets is $\alpha$-winning. 
\end{proposition}
For a proof of Proposition \ref{secgameprop2}, see e.g., \cite[Theorem 2]{schmidtsgame}. 

Winning properties in Schmidt's game also imply full Hausdorff dimensions in \textsl{Ahlfors regular} measure spaces. For a constant $s>0$, a probability measure $\mu$ of $X$ is called \textsl{Ahlfors $s$-regular} if there is positive constants $\rho_{\max}$ and $K$ such that for all $(z,\rho)\in X\times (0,\rho_{\max})$ 
\begin{equation*}
    \frac{1}{K}\rho^s \leq \mu\left(B(z,\rho)\right)\leq K\rho^s.
\end{equation*}
The space $X$ is called \textsl{Ahlfors $s$-regular} if it is equal to the support of an Ahlfors $s$-regular probability measure. It is very well known (see e.g., \cite[Proposition 5.1]{dkwbschmidtgame}) that 
\begin{proposition}
    Let $s>0$. Suppose $(X,d)$ is a Ahlfors $s$-regular metric space. If $S\subset X$ is winning, then the Hausdorff dimension of $S$ is $s$. 
\end{proposition}

McMullen introduced a variation of an infinite game like Schmidt's game in \cite{mcmullengame}. To define McMullen's game, let $\beta\in (0,1/3)$ be a constant. Again, consider two players $\pa$ and $\pb$. $\pb$ and $\pa$ take turns to pick elements $(z_k,r_k)$ and $(\zeta_k,\rho_k)$ of $X\times (0,\infty)$ respectively, such that 
\begin{equation*}
    B(z_1,r_1)\supset \left[B(\zeta_1,\rho_1)\backslash B(z_1,r_1)\right]\supset B(z_2,r_2) \supset \cdots B(z_k,r_k)\supset  \left[B(\zeta_k,\rho_k)\backslash B(z_k,r_k)\right]\supset B(z_{k+1},r_{k+1})\supset \cdots 
\end{equation*}
and 
\begin{equation*}
    \rho_k\leq \beta r_k\quad\text{and}\quad r_{k+1}\geq \beta r_k. 
\end{equation*}
A subset $S$ of $X$ is called \textsl{$\beta$-absolute winning} given a constant $\beta\in (0,1/3)$ if $\pa$ can always set up her plays, regardless of how $\pb$ plays, so that $\bigcap_{k=1}^\infty B(z_k,r_k)\cap S\neq \varnothing$.
\par

It is very obvious that winning in McMullen's game is stronger than winning in Schmidt's game, as in McMullen's game, one may view $\pa$'s moves as merely blocking one of $\pb$'s possible moves, so $\pa$ has very little control over the states of the game. In particular, 
\begin{proposition}
    If a subset $S$ is $\beta$-absolute winning in McMullen's game for all $\beta\in \left(0,\frac{1}{3}\right)$, then $S$ is $\alpha$-winning in Schmidt's game for all $\alpha\in \left(0,\frac{1}{2}\right)$. 
\end{proposition}

Therefore, from Theorem \ref{mainthm}, it immediately follows that 
\begin{corollary}
     In Systems I and II, for any equicontinuous sequence of functions $\left(f_n:[0,1]\to [0,1]\right)_n$, the set $\mathfrak{BA}\left(\left(f_n\right)_n\right)$ is $\alpha$-winning (in Schmidt's game) for all $\alpha\in \left(0,\frac{1}{2}\right)$.
\end{corollary}
Consequently, 
\begin{corollary}
     In Systems I and II, for any equicontinuous sequence of functions $\left(f_n:[0,1]\to [0,1]\right)_n$, the set $\mathfrak{BA}\left(\left(f_n\right)_n\right)$ has local Hausdorff dimension $1$.
\end{corollary}

\section{Proof of Theorem \ref{mainthm}}\label{secproof}

\subsection{Setup and inductive hypothesis}\label{secproof1}

Let $\beta\in \left(0,\frac{1}{3}\right]$. We will supply     $\pa$ with a winning strategy for each system.

First, we define the following terms: 
\begin{itemize}
    \item A point $z\in [0,1]$ is called a $n$th-order \textbf{vertex} (also denoted by vertex$^{(n)}$) if $z\in T^{-n}(0)$.
    \item An interval is called an $n$th-order \textbf{cylinder} (also denoted by cylinder$\level{n}$) if its endpoints are two consecutive (ordered by the natural order of real numbers) vertices$\level{n}$. 
    \begin{remark}
        Note that 
        \begin{itemize}
            \item $T^n$ is injective on any cylinder$\level{n}$, and
            \item given a subset of $[0,1]$, one can count how many cylinders$\level{n}$ it intersects by counting how many vertices$\level{n}$ are contained in it; for example, if a set does not contain any vertex$\level{n}$, then it is included in a cylinder$\level{n}$. 
        \end{itemize}
    \end{remark}
    \item For all $n\geq 1$, a point $z\in [0,1]$ is called a $n$th-order \textbf{*-vertex} (denoted by *-vertex$^{(n)}$) if $z\in T^{-(n-1)}\left(f_{n-1}\left(c^{(n-1)}\right)\right)$ where $c^{(n-1)}$ is the center of $B_1\level{n-1}$. 
    \item Denote the collection of the *-vertices of order less  than or equal to $n$ by $\mathcal{S}\level{n}$.

\item Denote the diameter of $B_k\level{n}$ by $\tau_k\level{n}$, denote the radius of $B_k\level{n}$ by $r_k\level{n}$, and denote the diameter of $T^{n-1}B_k\level{n}$ by $\rho_k\level{n}$. 
\end{itemize}

Denote the balls chosen by $\pb$ by $\underbrace{B_1^{(1)}}_{\subset [0,1]}\to  \underbrace{B_2^{(1)}}_{\subset B_1^{(1)}\backslash A_1^{(1)}}\to B_3^{(1)}\to \cdots $ and those chosen by $\pa$ by $\underbrace{A_1^{(1)}}_{\subset B_1^{(1)}}\to A_2^{(1)}\to \cdots $. We will allow \textit{zero and negative indices} for these balls; for example, $B_{0}\level{n}$ is the ball one round before $B_{1}\level{n}$ and $B_{-5}\level{n}$ is the ball six rounds before $B_{1}\level{n}$.

In both proofs, we will verify that the following hypotheses hold for all $n\in \mathbb{N}$.
Let $\rho>0$ and $\lambda>0$ be constants that will be specified in the strategies. 

\begin{hypothesis}\label{s1}
\begin{itemize}
    \item In System I, a natural number $n$ (which will be later referred to as the ``$n$th level'') can be categorized into one of the three types: 
\begin{itemize}
    \item \textbf{Type I:} Level $n$ is said to be of \textbf{Type I} if $B_1\level{n}$ does not contain any vertex$\level{n}$.
    \item \textbf{Type II:} Level $n$ is said to be of \textbf{Type II} if the interior of $B_1\level{n}$ does not contain any vertex$\level{n-1}$, $T^{n-1}B_1\level{n}$ contains exactly one vertex$\level{1}$ and $\diam\left(T^{n-1}B_1\level{n}\right)< \frac{\rho}{\gamma}$.
    \item \textbf{Type III:} Level $n$ is said to be of \textbf{Type III} if The interior of $B_1\level{n}$ does not contain any vertex$\level{n-1}$,  $T^{n-1}B_1\level{n}$ contains at least one vertex$\level{1}$ and $\diam\left(T^{n-1}B_1\level{n}\right)\geq \frac{\rho}{\gamma}$.
\end{itemize}

    \item In System II, level $n$ belongs to one of the three types: 
    \begin{itemize}
    \item \textbf{Type I:} Level $n$ is said to be of \textbf{Type I} if $B_1\level{n}$ does not contain any vertex$\level{n}$.
    \item \textbf{Type II:} Level $n$ is said to be of \textbf{Type II} if the interior of $B_1\level{n}$ does not contain any vertex$\level{n-1}$,  $T^{n-1}B_1\level{n}$ contains exactly one vertex$\level{1}$, and $\diam(T^{n-1}B_1\level{n})<\frac{\beta^2}{4(N\level{n})^2}$ where $\frac{1}{N\level{n}}$ is the vertex$\level{1}$ in $T^{n-1}B_1\level{n}$.  
    \item \textbf{Type III:} Level $n$ is said to be of \textbf{Type III} if the interior of $B_1\level{n}$ does not contain any vertex$\level{n-1}$,  $T^{n-1}B_1\level{n}$ contains at least one vertex$\level{1}$, and $\diam(T^{n-1}B_1\level{n})\geq \frac{\beta^2}{4(N\level{n})^2}$ where $\frac{1}{N\level{n}}$ is the rightmost vertex$\level{1}$ in $T^{n-1}B_1\level{n}$.  
\end{itemize}
\end{itemize}
\end{hypothesis}

Denote by $J\level{n}$ the last appearance of a Type-III level before level $n$; i.e., $$J\level{n}=\begin{cases}
        \max \left\{k:\, k<n\text{ and }\text{level}\level{k}\text{ is Type-III}\right\} & \text{if }\left\{k:\, k<n\text{ and }\text{level}\level{k}\text{ is Type-III}\right\}\neq \varnothing,\\
        0& \text{otherwise}. 
    \end{cases}$$

\begin{hypothesis}\label{ih.2}
   $\rho_1\level{n}\geq \begin{cases}
       \rho & \text{if there is a Type-II level after the last Type-III level and before Level }n; \\
       \beta\rho & \text{otherwise}. 
   \end{cases}$
\end{hypothesis}

\begin{hypothesis}\label{ih.3}
     For all Type-III levels $n,m$ and $m-n\geq \lambda$, there must exist some $n<k<m$ in such that level $k$ is of Type III.
\end{hypothesis}

\begin{lemma}\label{indlemma}
     Statements \ref{s1}, \ref{ih.2} and \ref{ih.3} hold for all natural numbers $m,n$. 
\end{lemma}

To prove Statement \ref{ih.3} holds for all $m,n$, we will need the following lemma: 
\begin{lemma}\label{lemmanoconsec2}
        For all natural numbers $n,m$ with $1\leq n<m$, if levels $n,m$ belong to Type II, then there exists some natural number $k$ in $(n,m)$ such that level $k$ is of Type III.
    \end{lemma}

We claim that under the given strategy, the following property holds: 
\begin{proposition}\label{Iresultprop}
    There exists some $\delta>0$ such that for all $n>1$ and for all $x\in B_{1}\level{n+1}$, one has that $$\dist\left(T^{n} x,f_n\left(c\level{n}\right)\right)\geq \delta.$$
\end{proposition}

\subsection{Construction of strategy, I}\label{secproof2}

Let $\rho=\beta\cdot \min \left\{\rho_1\level{1}, \rho_\#\right\}$, where  $\rho_\#$ is a positive solution to 
\begin{equation}\label{rhosys1}
    \frac{-\log\rho_\#-\log\beta}{\log\gamma}=-4,
\end{equation}
and $\lambda=\frac{-\log \rho-\log \beta}{\log \gamma}+1$. 

The base case for Lemma \ref{indlemma} easily holds.

Let $M\in \mathbb{N}\backslash \{0\}$. Assume the following: 
\begin{inductive}
    Statements \ref{s1}, \ref{ih.2} and \ref{ih.3} hold for all natural numbers $0<m,n<M$. 
\end{inductive}

Next, we state the strategy for $\pa$: 

\begin{framed}

 \textit{Case I:} If the $n$th level is of Type I, then $\pa$ relabels $B_1\level{n}$ by $B_1\level{n+1}$. Note that this ``action'' of relabeling by $\pa$ does not alter the state of the game.

 \textit{Case II:} If the $n$th level is of Type II, then $\pa$ takes $A_1\level{n}=B(v\level{n},\beta\cdot \tau_1\level{n})$, where $v\level{n}$ is the vertex$\level{n}$ in $B_1\level{n}$, and then $\pa$ relabels $B_2\level{n}$ by $B_1\level{n+1}$.

\textit{Case III:} Suppose that the $n$th level is of Type III. Denote the collection of the *-vertices of order less  than or equal to $n$ by $\mathcal{S}\level{n}$. We will see that by induction $\left|\mathcal{S}\level{n-1}\cap B_1\level{n}\right|\leq \lambda$. For all $k\geq 1$, \textsl{we ask $\pa$ to take the center of $A_k\level{n}$ to be $\begin{cases}
    \text{an element of }\mathcal{S}\level{n-1}\cap B_k\level{n}& \text{if }B_k\level{n}\cap \mathcal{S}\level{n-1}\neq\varnothing ,\\
    \text{the center of }B_k\level{n} & \text{if }B_k\level{n}\cap \mathcal{S}\level{n-1}=\varnothing, 
\end{cases}$ with the maximal allowed radii $\beta\cdot r_k\level{n}$}, until the time 
  $$H\level{n}:= \min \left\{k:\,
           B_k\level{n}\cap \mathcal{S}\level{n-1}=\varnothing \text{ and }
           \diam(B_k\level{n})<\frac{\rho}{\gamma}
        \right\};$$
        we denote the smallest $k$ such that $\diam(B_k\level{n})<\frac{\rho}{\gamma}$ by $L\level{1}$. Note that it takes at most $\lceil \frac{\log \lambda}{\log 2}\rceil$ rounds for Alice to remove all the points in $\mathcal{S}\level{n-1}\cap B_1\level{n}$, so 
        \begin{equation}\label{hnlnsys1}
          H\level{n}\leq L\level{n}+\lceil \frac{\log \lambda}{\log 2}\rceil.   
        \end{equation}
         Note that all cylinders contained $T^{n-1}B_1\level{n}\subset [0,1]$ have diameter $\frac{1}{\gamma}$, so $B_{H\level{n}}\level{n}$ contains at most one vertex$\level{n}$ and two *-vertex$\level{n}$. Alice remove these in two turns. Then 
 \begin{align*}
    \dist(T^{n-1}(B_{H\level{n}+2}\level{n}\backslash A_{H\level{n}+2}\level{n}),\mathcal{S}\level{n})\geq &  \radius (T^{n-1}A_{H\level{n}+2})\\
    \geq &  \beta^{\lceil \frac{\log \lambda}{\log 2}\rceil+3}\cdot \radius(T^{n-1}B_{L\level{n}-1}\level{n})\\
    \geq &  \frac{1}{2} \beta^{\lceil \frac{\log \lambda}{\log 2}\rceil+3}\cdot \frac{\beta}{2\gamma},
 \end{align*}
 so as a consequence, for all $x\in B_1\level{n+1}\subset B_{H\level{n}+2}\level{n}\backslash A_{H\level{n}+2}\level{n})$, one has 
 \begin{equation}\label{Iresult!}
     \dist(T^k x,f\left(c\level{k}\right))\geq \gamma^{-\lambda}\frac{1}{2} \beta^{\lceil \frac{\log \lambda}{\log 2}\rceil+3}\cdot \frac{\beta}{2\gamma}\quad \text{for all }J\level{n}<k\leq n.
 \end{equation}
Additionally, we have 
\begin{align}
   \rho_1\level{n+1}=& \diam \left(T^nB_1\level{n+1}\right)\notag \\
   =& \gamma^n \cdot \diam \left(B_1\level{n+1}\right)\notag \\
   \geq & \gamma^n \cdot \beta \cdot \diam \left(B_{H\level{n}+2}\level{n}\right)\notag \\
   \overset{\eqref{hnlnsys1}}{\geq} & \gamma^n \cdot \beta \cdot \diam \left(B_{\left(L\level{n}-1\right)+\lceil \frac{\log \lambda}{\log 2}\rceil+3}\right)\notag \\
   \geq & \gamma^n \cdot \beta^{\lceil \frac{\log \lambda}{\log 2}\rceil+4}\cdot  \diam\left(B_{L\level{n}-1}\level{n}\right)\notag \\
   =&  \gamma \cdot \beta^{\lceil \frac{\log \lambda}{\log 2}\rceil+4}\cdot  \diam\left(T^{n-1}B_{L\level{n}-1}\level{n}\right)\notag \\
   \geq &  \gamma \cdot \beta^{\lceil \log_2\frac{-\log\rho-\log\beta}{\log\gamma}\rceil+4}\cdot \frac{\rho}{\gamma}\notag \\
   \overset{\eqref{rhosys1}}{=}& \rho. \label{Iradiusproj}
\end{align}

\end{framed}

Now we verify our hypotheses hold in the progression. 

\begin{proof}[Proof of Lemma \ref{indlemma} for System I]
    For Statement \ref{s1}, note that if a ball $B$ contains two vertices$\level{n}$, then the diameter of $T^{n-1}B$ is at least $\frac{1}{\gamma}$, so if $B_1\level{n}$ contains a vertex$\level{n}$ (not of Type I), then either it contains exactly one vertex$\level{n}$ (of Type II) or $T^{n-1}B_1\level{n}\geq \frac{1}{\gamma}$ (of Type III). \par 

By assuming Statements \ref{ih.2} and \ref{ih.3} for $n<M$, if Level $(M-1)$ is Type-III, then Statement \ref{ih.2} for $n=M$ follows immediately from \eqref{Iradiusproj}. 
 Recall that $$J\level{n}=\begin{cases}
        \max \left\{k:\, k<n\text{ and }\text{level}\level{k}\text{ is Type-III}\right\} & \text{if }\left\{k:\, k<n\text{ and }\text{level}\level{k}\text{ is Type-III}\right\}\neq \varnothing,\\
        0& \text{otherwise}. 
    \end{cases}$$

    We now prove Lemma \ref{lemmanoconsec2} for this system and then use it to show Statement \ref{ih.3} holds. 
      
    \begin{proof}[Proof of Lemma \ref{lemmanoconsec2} for System I]
    Assume the contrary; suppose there are two consecutive Type-II levels $n,m$ with $n<m$ without any intermediate Type-III level. Then $B_1\level{n}$ contains a vertex$\level{n}$ $v_1$ and a vertex$\level{m}$ $v_2$, which are distinct. Then $$\diam\left(B_1\level{m}\right)=\beta\cdot \diam\left(B_1\level{n}\right)\geq \beta\cdot \dist\left(v_1,v_2\right)\geq \beta\cdot \frac{1}{\gamma^m},$$ contradicting that Level $m$ is Type-II. 
\end{proof}
    Then by Lemma \ref{lemmanoconsec2}, there is at most one Type-II level between Level $J\level{M}$ and Level $M$. Then 
    \begin{align*}
      \rho_1\level{M}= \diam\left(T^{M-1}B_1\level{M}\right)\geq \beta\cdot \diam\left(T^{M-1}B_1\level{J\level{M}+1}\right)\geq \beta\cdot \diam\left(T^{J\level{M}}B_1\level{J\level{M}+1}\right)\overset{\eqref{Iradiusproj}}{\geq}  \beta\cdot \rho. 
    \end{align*}

    \par

Since Statement \ref{ih.3} is assumed for all $0<n<M$, by the expanding property of $T$, note that for all $0<n<M$, 
    \begin{equation*}
        \diam(T^{n+\lambda-1}B_1\level{n+\lambda})\geq \beta \cdot \gamma^{\lambda-1} \diam(T^n B_1\level{n})=\frac{1}{\gamma}\cdot \frac{\gamma}{\rho\cdot \beta}\cdot \rho=1,
    \end{equation*}
    then $T^{\lceil\lambda\rceil} B_1\level{n}=[0,1]$ since its diameter exceeds $1$, meaning that within $\lambda$ levels any Type-III level $(n)$, there must be another Type-III level and proving Statement \ref{ih.3} for all $n$. 
\end{proof}

From \eqref{Iresult!}, we finally obtain the most crucial result: Proposition \ref{Iresultprop}. We will finish our proof of Theorem \ref{mainthm} in Section \ref{finishproof}.

\subsection{Construction of strategy, II}\label{secproof3}
\label{proofmainthmii}

Before we start our statement of the strategy, we state two well-known properties of System II. 

\begin{lemma}[Expanding property for System II]\label{expandii}
 There exists a constant $R>0$ such that for all positive $n\in \mathbb{N}$ and for all $x\in (0,1)$, 
    \begin{equation*}
      |D_xT^n|\geq R \cdot \left[\left(\frac{1+\sqrt{5}}{2}\right)^2\right]^n.
    \end{equation*}
\end{lemma}

For the remaining of the paper, we denote $\psi:=\left(\frac{1+\sqrt{5}}{2}\right)^2$. 

\begin{proof}
    From \cite[Chapter 1, \S2]{khinchincontinuedfrac}, in view of the notations in this book, the denominators of the congregants satisfy the recurrence relation: 
    \begin{equation*}
    \begin{cases}
        q_0=1,q_1=a_1,&\\
        q_n=a_nq_{n-1}+q_{n-2}& \forall \text{ integers } n>1.
        \end{cases}
    \end{equation*}
    together with the facts that $a_n\geq 1$ for all $n$, so $q_n$ is greater than or equal to the $n$th Fibonacci number, and hence $q_n\geq \psi^{n/2}$. By the Mean Value Theorem, diameter of a cylinder$\level{n}$ cannot be less than $\frac{1}{|D_xT^n|}$. 
    The diameter of a cylinder$\level{n}$ is $\frac{1}{q_n\left(q_n+q_{n-1}\right)}$, so the lemma is proved. 
\end{proof}

\begin{lemma}[Bounded distortion for System II]\label{bdii}
    For all positive $n\in \mathbb{N}$ and for all $x,y$ in an $n$th order cylinder, 
    \begin{equation*}
        \frac{1}{4}\leq \frac{|D_xT^n|}{|D_yT^n|}\leq 4.
    \end{equation*}
\end{lemma}

Lemma \ref{bdii} is also well known (see e.g., \cite[Section 7.4, Lemma 2]{ergodictheorybookcfs}).

Let $\rho_\#,\lambda$ be positive solutions to the system 
\begin{equation}\label{IIrholambda}
    \begin{cases}
       \left(\beta/4\right)^{\lceil \frac{\log \lambda}{\log 2}\rceil+6}\geq 2\rho_\#,\\
       \frac{R}{4}\beta \psi^{\lambda-1}\geq \frac{1}{\rho_\#}.
    \end{cases}
\end{equation} One can simply take $\lambda$ large enough such that 
\begin{equation*}
    \frac{4}{R\beta}\cdot \left(\frac{1}{\psi}\right)^{\lambda-1}\leq \frac{1}{2}\cdot \left(\beta/4\right)^{\frac{\log\lambda}{\log2}+6},
\end{equation*}
and $\rho_\#$ to be any number in $\left(\frac{4}{R\beta}\cdot \psi^{1-\lambda},\frac{1}{32}\cdot \left(\beta/4\right)^{7+\frac{\log\lambda}{\log2}}\right)$.
Let $\rho=\min \left\{\rho_1\level{1},\rho_\#\right\}$. 

The base case for Lemma \ref{indlemma} easily holds.

Let $M\in \mathbb{N}\backslash \{0\}$. Assume the following: 
\begin{inductive}
    Statements \ref{s1}, \ref{ih.2} and \ref{ih.3} hold for all natural numbers $0<m,n<M$. 
\end{inductive}

Next, we state the strategy for $\pa$: 

\begin{framed}
    \textit{Case I:} If the $n$th level is of Type I, then $\pa$ relabels $B_1\level{n}$ by $B_1\level{n+1}$.

 \textit{Case II:} If the $n$th level is of Type II, then $\pa$ takes $A_1\level{n}=B(v\level{n},\beta\cdot r_1\level{n})$, where $v\level{n}$ is the vertex$\level{n}$ in $B_1\level{n}$ and $r_1\level{n}$ is the radius of $B_1\level{n}$, and then $\pa$ relabels $B_2\level{n}$ by $B_1\level{n+1}$. 

\textit{Case III:} Suppose that the $n$th level is of Type III.  \par 
$\pa$ first select $A_1\level{n}$ with radius $\beta\cdot \radius(B_1\level{n})$ so that $T^{n-1}A_1\level{n}$ is placed as left as possible in $T^{n-1}B_1\level{n}$. Note that $\rho_1\level{n} \begin{cases}
    <\frac{1}{N\level{n}-1} & \text{if }N\level{1}\neq 1,\\
    \leq \frac{1}{N\level{n}} & \text{if }N\level{1}=1; 
\end{cases}$ in particular, we can conclude that 
\begin{equation}\label{r1nest}
    \rho_1\level{n}<\frac{1}{N\level{n}}.
\end{equation}
In Case III, we need to consider the following: 
\begin{itemize}
    \item \textit{Subcase 1}: \textit{\underline{Suppose $B_1\level{n}\backslash A_1\level{n}$ contains zero or one vertex$\level{n}$, say $v\level{n}$.}} Then $B_2\level{n}$ contains at most $\lambda$ *-vertices from $\mathcal{S}\level{n-1}$, where we denote the collection of the *-vertices of order less than or equal to $n$ by $\mathcal{S}\level{n}$. By induction, $\left|\mathcal{S}\level{n-1}\cap B_1\level{n}\right|\leq \lambda$. $\pa$ then selects each $A_k\level{n}$ in $A_2\level{n},\ldots,A_{1+\lceil\log_2(\lambda+2)\rceil}\level{n}$ with the maximal radii allowed and first centered at $v\level{n}$ and then centered at the *-vertices in $\mathcal{S}\level{n-1}\cap B_2\level{n}$ if $\mathcal{S}\level{n-1}\cap B_k\level{n}\neq \varnothing$, and centered at the center of $B_k\level{n}$ if $\mathcal{S}\level{n-1}\cap B_k\level{n}=\varnothing$. Consequently, for all $J\level{n}<k<n$ where $J\level{n}$ is the last appeared Type III level, we have 
    \begin{align}
        \dist \left(T^{k-1}B_1\level{n+1},T^{-1}(f_k(c\level{k}))\right)=& \dist \left(T^{k-1}B_{2+\lceil\log_2(\lambda+2)\rceil}\level{n},T^{-1}(f_k(c\level{k}))\right)\\
        \geq & \frac{1}{4}\cdot \radius(T^{k-1}A_{1+\lceil\log_2(\lambda+2)\rceil}\level{n})\notag \\
        \geq & \frac{1}{4}\cdot \beta^{1+\lceil\log_2(\lambda+2)\rceil}\cdot \rho_1\level{k} \geq\frac{1}{4}\cdot \beta^{1+\lceil\log_2(\lambda+2)\rceil}\cdot \rho. \notag 
    \end{align}
    Hence consequently, for all $x\in B_1\level{n}=B_{2+\lceil\log_2(\lambda+2)\rceil}\level{n}$, one has 
    \begin{equation}\label{IIresult!1}
        \dist \left(T^k x, f_n\left(c\level{k}\right)\right)\geq \frac{1}{4}\cdot \beta^{1+\lceil\log_2(\lambda+2)\rceil}\cdot \rho. 
    \end{equation}
    Then $\pa$ relabels $B_{2+\lceil\log_2(\lambda+2)\rceil}\level{n}$ by $B_1\level{n+1}$. Note that 
    \begin{align}
        \diam(T^{n-1}B_1\level{n+1})=&\diam \left(T^{n-1}B_{2+\lceil\log_2(\lambda+2)\rceil}\level{n}\right)\notag \\
        \geq & \diam(T^{n-1}B_1\level{n})\cdot \beta^{2+\lceil\log_2(\lambda+2)\rceil}\notag \\
        \geq & \frac{1}{4\left(N\level{n}\right)^2}\cdot \beta^{2+\lceil\log_2(\lambda+2)\rceil}\label{iii-1-lb1}
    \end{align}
    On the other hand, since the ball $T^{n-1}B_1\level{n+1}$ does not contain any vertices$\level{n}$, $T^{n-1}B_1\level{n+1}$ is included in a cylinder $J\level{1}$ of order $1$. Since $J\level{1}$ is to the left of the cylinder$\level{1}$ $\left[\frac{1}{N\level{n}},\frac{1}{N\level{n}-1}\right]$, its diameter will not exceed that of $\left[\frac{1}{N\level{n}},\frac{1}{N\level{n}-1}\right]$; hence $\diam\left(J\level{1}\right)\leq \frac{1}{N\level{n}\left(N\level{n}-1\right)}\leq \frac{2}{\left(N\level{n}\right)^2}$, so together with the bounded distortion property (Lemma \ref{bdii}) and \eqref{iii-1-lb1}, we have that 
    \begin{align}
     \rho_1\level{n}=\frac{\diam(T^{n}B_1\level{n+1})}{\diam\left([0,1]\right)}\geq &  \left[\frac{\diam \left(T^{n-1}B_1\level{n+1}\right)}{\diam\left(J\level{n}\right)}\right]\cdot \frac{1}{4}\notag \\
     \overset{\eqref{iii-1-lb1}}{\geq} &\left[\frac{1}{4\left(N\level{n}\right)^2}\cdot \beta^{2+\lceil\log_2(\lambda+2)\rceil} \middle/ \diam\left(J\level{1}\right)\right]\cdot \frac{1}{4}\notag \\
       \geq & \left[\frac{1}{4\left(N\level{n}\right)^2}\cdot \beta^{2+\lceil\log_2(\lambda+2)\rceil} \middle/ \frac{2}{\left(N\level{n}\right)^2}\right]\cdot \frac{1}{4} \\
       =&\frac{\beta^{2+\lceil\log_2(\lambda+2)\rceil}}{32}\notag\\
       \overset{\eqref{IIrholambda}}{\geq}& \rho.\label{IIradiusproj1}
    \end{align}
\item \textit{Subcase 2:} \textit{\underline{Suppose $B_1\level{n}\backslash A_1\level{1}$ contains more than one vertex$\level{n}$.}} Let $\frac{1}{K\level{n}}$ be the leftmost vertex$\level{1}$ in $T^{n-1}(B_1\level{n}\backslash A_1\level{n})$. Since $T^{n-1}(A_1\level{n})$ is the leftmost subinterval of $T^{n-1}(B_1\level{n})$ and the length of $T^{n-1}(A_1\level{n})$ is at least $\frac{\beta}{4}\cdot \rho_1\level{n}$ by Lemma \ref{bdii}, so $\frac{1}{K\level{1}}>\frac{\beta}{4}\cdot \rho_1\level{n}$. Consequently, we have 
    \begin{align*}
        \frac{1}{N\level{n}}-\frac{1}{K\level{n}}< \diam(T^{n-1}B_1\level{n})=\rho_1\level{n}<\frac{\beta}{4}\cdot \frac{1}{K\level{n}},
    \end{align*}
    so 
    \begin{align}\label{5.3.eqcompareNK}
        \frac{1}{N\level{n}}<(1+\beta/2)\cdot \frac{1}{K\level{n}}.
    \end{align}
    Since $B_1\level{n}$ does not contain any vertex$\level{n-1}$, it contains at most $\lambda$ ($\lambda$ is defined in Lemma) *-vertices of order less than $n$. 
    Then $\pa$ takes the radius of $A_k\level{n}$ to be $\beta\cdot r_k\level{n}$ (maximal allowed radius) and its center to be 
    \begin{align*}
        \begin{cases}
    \text{an element of }\mathcal{S}\level{n-1}\cap B_k\level{n}& \text{if }B_k\level{n}\cap \mathcal{S}\level{n-1}\neq\varnothing ,\\
    \text{the center of }B_k\level{n} & \text{if }B_k\level{n}\cap \mathcal{S}\level{n-1}=\varnothing, 
\end{cases}
    \end{align*}
    until the time (until $\pa$'s turn $A_{H\level{n}}\level{n}$)
    \begin{align}\label{5.3.eqevent}
      H\level{n}:= \min \left\{k:\,
            k> \left|\mathcal{S}\level{n}\right|+1 \text{ and }
            \diam(T^{n-1}B_k\level{n})<\frac{\beta^2}{4(N\level{n})^2}
        \right\};
    \end{align}
    we denote the smallest $k$ for which $\diam(T^{n-1}B_k\level{n})<\frac{\beta^2}{4(N\level{n})^2}$ by $L\level{n}$. It is easy to observe that $H\level{n}\leq L\level{n}+\lceil\frac{\log\lambda}{\log 2}\rceil$. Additionally, for all cylinders$\level{1}$ $J$ contained in $T^{n-1}(B_1\level{n}\backslash A_1\level{n})$, the cylinder $J$ is contained in $\left[\frac{1}{K\level{n}},\frac{1}{N\level{n}}\right]$. By \eqref{5.3.eqcompareNK}, in particular, $\frac{1}{N\level{n}}<\left(1+\beta/2\right)\cdot \frac{1}{K\level{n}}<(1+1)\cdot \frac{1}{K\level{n}}$ so  $\left[\frac{1}{K\level{n}},\frac{1}{N\level{n}}\right]\subset \left[\frac{1}{2N\level{n}},\frac{1}{N\level{n}}\right]$, and hence 
    \begin{equation*}
        \frac{1}{2N\level{n}}\cdot \frac{1}{2N\level{n}}<\diam(J)<\frac{1}{N\level{n}}\cdot \frac{1}{N\level{n}}, 
    \end{equation*}
    and hence $B_{L\level{n}}\level{n}$ does not contain any cylinder$\level{n}$. Therefore, $B_{L\level{n}+\lceil\frac{\log\lambda}{\log 2}\rceil}\level{n}$ contains at most one vertex$\level{n}$, say $v\level{n}$, and at most two *-vertices$\level{n}$, say $s_1\level{n},s_2\level{n}$, which Alice will removes them from the game as centers of $A_{H\level{n}+1}$ and $A_{H\level{n}+2}$; i.e., $A_{H\level{n}+1}=B(v\level{n},\beta\rho_{H\level{n}+1}\level{n})$ and $A_{H\level{n}+2}=B(s_i\level{n},\beta\rho_{H\level{n}+2}\level{n})$ where $s_i\level{n}\in B_{H\level{n}+1}\backslash A_{H\level{n}+1}$. \par

  Then $\pa$ relabels $B_{H\level{n}+3}\level{n}$ by $B_1\level{n+1}$. We find
    \begin{align*}
    \dist\left(T^{n-1}B_1\level{n+1},f_n(c\level{n})\right)\geq &
        \dist\left(T^{n-1}\left(B_{H\level{n}+2}\level{n}\backslash A_{H\level{n}+2}\level{n}\right),f_n\left(c\level{n}\right)\right)\\
        \geq & \frac{1}{4}\cdot  \radius\left(T^{n-1}A_{L\level{n}+\lceil\frac{\log\lambda}{\log 2}\rceil+2}\right)\\
        \geq & \beta^{\lceil\frac{\log\lambda}{\log 2}\rceil+4}\cdot \radius\left(T^{n-1}B_{L\level{n}-1}\level{n}\right)\\
        \geq & \beta^{\lceil\frac{\log\lambda}{\log 2}\rceil+4}\cdot \frac{\beta^2}{8(N\level{n})^2}\\
        \overset{\eqref{r1nest}}{\geq}& \frac{\beta^{\lceil\frac{\log\lambda}{\log 2}\rceil+6}}{8}\cdot (\rho_1\level{n}/2)^2\\
        \overset{\text{Statement \ref{ih.2}}}{\geq }& \frac{\beta^{\lceil\frac{\log\lambda}{\log 2}\rceil+6}}{8}\cdot (\rho\beta/2)^2.
    \end{align*}
Additionally, for all $J\level{n}<k< n$, it holds that 
\begin{align*}
  \dist\left(T^{k-1}B_1\level{n+1},f_k\left(c\level{k}\right)\right) \geq & \dist\left(T^{k-1} B_{\lceil\log\lambda/\log 2\rceil}\level{J\level{n}+1},f_k\left(c\level{k}\right)\right)\\
  \geq & \radius\left(T^{k-1}A_{\lceil\log\lambda/\log 2\rceil+1}\level{J\level{n}+1}\right)\\
  \geq & \frac{1}{4}\cdot \radius \left(T^{k-1}A_1\level{J\level{n}+1}\right)\cdot \beta^{\lceil\log\lambda/\log 2\rceil}\\
  \geq & \frac{1}{4}\cdot \radius \left(T^{J\level{n}}A_1\level{J\level{n}+1}\right)\cdot \beta^{\lceil\log\lambda/\log 2\rceil}\\
  \geq & \frac{1}{16}\cdot \radius \left(T\level{J\level{n}}B_1\level{J\level{n}+1}\right)\cdot \beta^{\lceil\log\lambda/\log 2\rceil+1}\\
  \overset{\text{Statement \ref{ih.2}}}{\geq} & \frac{1}{16}\cdot \beta\rho\cdot \beta^{\lceil\log\lambda/\log 2\rceil+1}. 
\end{align*} 
Consequently, for all $x\in B_1\level{n+1}$ and for all $J\level{n}<k\leq  n$, we have 
\begin{equation}\label{IIresult!}
    \dist\left(T^{k-1}x,f_k\left(c\level{k}\right)\right)\geq \min \left\{\frac{\beta^{\lceil\frac{\log\lambda}{\log 2}\rceil+6}}{8}\cdot (\rho\beta/2)^2,\frac{1}{16}\cdot \beta\rho\cdot \beta^{\lceil\log\lambda/\log 2\rceil+1}\right\}. 
\end{equation}
\par
Finally, note that 
\begin{align*}
   \rho_1\level{n+1}=& \diam\left(T^nB_1\level{n+1}\right)\\
   \geq & \frac{1}{4}  \cdot \beta\cdot   \diam\left(T^{n}B_{H\level{n}+2}\right)\\
   =& \frac{1}{4} \cdot \beta \cdot \diam \left(T^{n}B_{\left(L\level{n}-1\right)+\lceil \frac{\log \lambda}{\log 2}\rceil+3}\right)\\
   \geq & \frac{1}{4} \cdot \beta \cdot \diam\left(B_{\left(L\level{n}-1\right)+\lceil \frac{\log \lambda}{\log 2}\rceil+3}\right)\cdot \frac{1}{R}\cdot \psi^{n-1}\\
\geq &4^{-\lceil \frac{\log \lambda}{\log 2}\rceil+3}\cdot \beta^{\lceil \frac{\log \lambda}{\log 2}\rceil+4}\cdot  \diam\left(T^{n}B_{L\level{n}-1}\level{n}\right).
\end{align*}
Note that $T^{n-1}B_{L\level{n}-1}\level{n}$ is included in a cylinder $J\level{1}$ of order $1$ with diameter not exceeding that of $\left[\frac{1}{N\level{n}},\frac{1}{N\level{n}-1}\right]$ and $\diam\left(J\level{1}\right)=\frac{1}{N\level{n}\left(N\level{n}-1\right)}\leq \frac{2}{\left(N\level{n}\right)^2}$. Hence by Lemma \ref{bdii}, we have 
\begin{align}
    \rho_1\level{n+1}
    =& \frac{\diam\left(T^n B_1\level{n+1}\right)}{\diam\left([0,1]\right)}\notag \\
    \geq & \left[\diam\left(T^{n-1}B_1\level{n+1}\right)\middle/\diam\left(J\level{1}\right)\right]\cdot \frac{1}{4}\notag \\
    \geq & \left[\left(1/4\right)^{\lceil \frac{\log \lambda}{\log 2}\rceil+4}\cdot \beta^{\lceil \frac{\log \lambda}{\log 2}\rceil+4}\cdot \diam\left(T^{n-1}B_{L\level{n}-1}\level{n}\right)\middle/\diam\left(J\level{1}\right)\right]\cdot \frac{1}{4}\notag \\
    \geq & \left[\frac{\beta^2}{4\left(N\level{n}\right)^2}\middle/\diam\left(J\level{1}\right)\right]\cdot \left(1/4\right)^{\lceil \frac{\log \lambda}{\log 2}\rceil+4}\cdot \beta^{\lceil \frac{\log \lambda}{\log 2}\rceil+4}\cdot \frac{1}{4}\notag \\
    \geq & \left[\frac{\beta^2}{4\left(N\level{n}\right)^2}\middle/\frac{2}{\left(N\level{n}\right)^2}\right]\cdot \left(1/4\right)^{\lceil \frac{\log \lambda}{\log 2}\rceil+4}\cdot \beta^{\lceil \frac{\log \lambda}{\log 2}\rceil+4}\cdot \frac{1}{4}\\
    =& \frac{1}{2}\cdot 4^{-\lceil \frac{\log \lambda}{\log 2}\rceil-6}\cdot \beta^{\lceil \frac{\log \lambda}{\log 2}\rceil+6}\notag \\
    \overset{\eqref{IIrholambda}}{\geq} & \rho. \label{IIradiusproj2}
\end{align}
\end{itemize}
\end{framed}

Now we verify our hypothesis hold in the progression.

 We first prove Lemma \ref{lemmanoconsec2} for this system and then use it to show Statement \ref{ih.3} holds. 

\begin{proof}[Proof of Lemma \ref{indlemma} for System II]

    For Statement \ref{s1}, note that if a ball $B$ contains two vertices$\level{n}$, $T^{n-1}B$ contains a cylinder of order $n$, so $\diam\left(T^{n-1}B\right)\geq \frac{1}{N\level{n}\left(N\level{n}+1\right)}$ where $\frac{1}{N\level{n}}$ is the rightmost vertex$\level{1}$ in $T^{n-1}B$. Hence, if $B_1\level{n}$ is not of Type I, then either it is Type II, or $\diam\left(T^{n-1}B_1\level{n}\right)\geq \frac{1}{N\level{n}\left(N\level{n}+1\right)}\geq \frac{\beta^2}{4\left(N\level{n}\right)^2}$  (of Type III). \par 
    By assuming Statements \ref{ih.2} and \ref{ih.3} for $n<M$, if Level $(M-1)$ is Type-III, then Statement \ref{ih.2} for $n=M$ follows immediately from \eqref{IIradiusproj1} and \eqref{IIradiusproj2}. Next, we show that Lemma \ref{lemmanoconsec2} also holds for System II.

    \begin{proof}[Proof of Lemma \ref{lemmanoconsec2} for II]
        Assume the contrary; suppose there are two consecutive Type-II levels, say $n$ and $m$ with $n<m$, without any intermediate Type-III level. Since $n,m$ are consecutive Type-II levels without any intermediate Type-III level, we know that $B_2\level{n}=B_1\level{n+1}=B_1\level{m}$.
        Let $v_1$ be the vertex$\level{n}$ in $B_1\level{n}$ and let $v_2$ be the vertex$\level{m}$ in $B_1\level{m}$; denote $N_1:=\frac{1}{T^{n-1}v_1}$ and $N_2:=\frac{1}{T^{m-1}v_2}$.
        Then we have two cases to consider, depending on which connected component of $B_1\level{n}\backslash A_1\level{n}=B_1\level{n}\backslash B\left(v_1,\beta\cdot \frac{\rho_1\level{n}}{2}\right)$ Bob picks to place $B_2\level{n}$. 
        \begin{itemize}
            \item \textbf{Case 1: Suppose that $T^{n-1}B_2\level{n}$ contains a neighborhood of $0$.} Then
\begin{align*}
            \diam\left(T^{m-1}B_1\level{m}\right)\geq & \diam\left(\left(0,\frac{1}{N_2}\right]\right)\cdot \frac{\beta}{4} =\frac{1}{N_2}\cdot \frac{\beta}{4}, 
        \end{align*}
        but this contradicts that level $m$ is of Type II. 
            \item \textbf{Case 2: Suppose that $T^{n-1}B_2\level{n}$ contains a neighborhood of $1$.} Then  
            \begin{align*}
            \diam\left(T^{m-1}B_1\level{m}\right)\geq & \diam\left(\left[\frac{1}{N_2},1\right)\right)\geq \diam\left(\left[\frac{1}{2},1\right)\right)\cdot \frac{\beta}{4} \geq \frac{1}{N_2}\cdot \frac{\beta}{4}, 
        \end{align*}
        but this contradicts that level $m$ is of Type II.
        \end{itemize}
        Therefore, we have proved the lemma by contradiction. 
    \end{proof}

    Then by Lemma \ref{lemmanoconsec2} for System II, there is at most one Type-II level between level $J\level{M}$ and level $M$, so 
    \begin{equation*}
        \rho_1\level{M}\geq \beta\cdot \diam \left(T^{J\level{M}}B_1\level{J\level{M}+1}\right)\overset{\eqref{IIradiusproj1}, \eqref{IIradiusproj2}}{\geq} \beta\cdot \rho. 
    \end{equation*}\par

Since Statement \ref{ih.3} is assumed for all $0<n<M$, by the expanding property of $T$ (Lemma \ref{expandii}), note that for all $0<n<M$,
\begin{equation*}
        \diam(T^{n+\lambda-1}B_1\level{n+\lambda})\geq \beta \cdot \frac{1}{4}\cdot R\cdot \psi^{\lambda-1}\cdot  \diam\left(T^n B_1\level{n}\right)\geq \beta \cdot \frac{1}{4}\cdot R\cdot \psi^{\lambda-1}\cdot \rho\geq 1. 
    \end{equation*}
    meaning that $T^{\lceil\lambda\rceil} B_1\level{n}=[0,1]$ by its diameter, so within $\lambda$ levels any Type-III level $(n)$, there must be another Type-III level and proving Statement \ref{ih.3} for all $n$. 
\end{proof}

From \eqref{IIresult!1} and \eqref{IIresult!}, again we obtain the most crucial result: Proposition \ref{Iresultprop}. We will finish our proof of Theorem \ref{mainthm} in the next section. 

    \subsection{Completing the proofs}\label{finishproof}

  We are now ready to prove Theorem \ref{mainthm}.
  
  Let $\{w\}=\bigcap_{k,n} B_{k}\level{n}$. Then by Proposition \ref{Iresultprop}, there exists some $\delta>0$ such that 
  \begin{equation*}
      \dist\left(T^n w,c\level{n}\right)\geq \delta \quad \forall n\in \mathbb{N}. 
  \end{equation*}
  On the other hand, since the sequence $\left(f_n\right)_n$ is equicontinuous, there exists some $\delta'>0$ so that as long as $|x-y|<\delta'$, $\left|f_n(x)-f_n(y)\right|<\frac{\delta}{2}$.
  
 Since the sequence $\left(c^{n}\right)_n$ converges to $w$, we may find $N$ large enough such that for all $n\geq N$, we have $\dist\left(c\level{n},w\right)<\delta'$. Hence for all $n\geq N$
 \begin{equation*}
      \dist\left(T^n w,f_n(w)\right)\geq \dist\left(T^n w,f_n\left(c\level{n}\right)\right)-\dist\left(f_n\left(c\level{n}\right),f_n(w)\right) \geq \frac{\delta}{2} \quad \forall n\in \mathbb{N}. 
  \end{equation*}

\bibliographystyle{alpha}
\bibliography{mybib}

\end{document}